\documentclass[a4paper]{article}

\usepackage{a4wide}
\usepackage{amsfonts}
\usepackage{amsthm}
\usepackage{xcolor}
\usepackage{mathtools}
\usepackage{amssymb}
\usepackage{amsfonts}
\usepackage{amsxtra}
\usepackage{amstext}
\usepackage{amsbsy}
\usepackage{amscd}
\usepackage{graphicx}
\usepackage{float}
\usepackage{xcolor}
\usepackage{calc}

\newcommand{\R}{\mathbb{R}}
\newcommand{\Q}{\mathbb{Q}}
\newcommand{\N}{\mathbb{N}}
\newcommand{\Z}{\mathbb{Z}}
\newcommand{\mx}[1]{[#1]}

\newtheorem{theorem}{Theorem}
\newtheorem{lemma}[theorem]{Lemma}

\newtheorem{remark}[theorem]{Remark}

\title{Bilevel linear optimization belongs to NP and\\ admits polynomial-size KKT-based reformulations}

\date{\today}

\author{Christoph Buchheim, Fakult\"at f\"ur Mathematik,\\ Technische Universit\"at Dortmund, Germany\\ \texttt{christoph.buchheim@math.tu-dortmund.de}}

\begin{document}

\maketitle

\begin{abstract}
  It is a well-known result that bilevel linear optimization is
  NP-hard. In many publications, reformulations as mixed-integer
  linear optimization problems are proposed, which suggests that the
  decision version of the problem belongs to NP. However, to the best
  of our knowledge, a rigorous proof of membership in NP has never
  been published, so we close this gap by reporting a simple but not
  entirely trivial proof. A related question is whether a large enough
  ``big M'' for the classical KKT-based reformulation can be computed
  efficiently, which we answer in the affirmative. In particular, our
  big M has polynomial encoding length in the original problem data.
\end{abstract}

\section{Introduction}

We consider a general bilevel optimization problem with linear
constraints at both the upper and lower level. Given as input the data
$c\in\Q^n$, $d,q\in\Q^m$, $A\in\Q^{k\times n}$, $B\in\Q^{k\times m}$,
$a\in\Q^k$, $T\in\Q^{r\times n}$, $W\in\Q^{r\times m}$, and
$h\in\Q^r$, such a problem can be written as
\begin{equation}\label{bpp}\tag{BLP}
\begin{array}{ll}
  \min & \, c^\top x+d^\top y\\
  \textnormal{~s.t.} & Ax+By = a,~x\ge 0\\
  & y \in \text{argmin}\;\{ q^\top \bar y\mid Tx+W\bar y=h,~\bar y\ge 0\}\;.
\end{array}
\end{equation}
In this formulation, we implicitly assume the \emph{optimistic}
scenario: whenever the lower level problem
\begin{equation}\label{bpp-ll}\tag{LL}
\begin{array}{ll}
  \min & \, q^\top y\\
  \textnormal{~s.t.} & T\bar x+Wy=h,~y\ge 0\;,
\end{array}
\end{equation}
for some given upper level choice~$\bar x\ge 0$, has more than one
optimal solution satisfying the \emph{coupling constraints}~$A\bar
x+By=a$, we assume that one minimizing~$d^\top y$ is chosen, i.e., a
best possible one for the upper level problem. In the
\emph{pessimistic scenario}, it is instead assumed that the chosen~$y$
\emph{maximizes}~$d^\top y$ over all optimal solutions
of~\eqref{bpp-ll}. For a further discussion of the optimistic and the
pessimistic scenario as well as structural properties of~\eqref{bpp}
and~\eqref{bpp-ll}, as a parametric optimization problem, we refer
to~\cite[Chapters~2 and~3]{Dempe} and the references therein.

In the optimistic formulation~\eqref{bpp}, any~$\bar x\ge 0$ that
renders the lower level problem~\eqref{bpp-ll} infeasible or unbounded
is an infeasible choice for the upper level by definition, and the
same is true when no optimal solution of~\eqref{bpp-ll} satisfies the
coupling constraints. In the pessimistic case, it is natural to assume
that~$\bar x$ is feasible only if \emph{all} optimal choices of
~\eqref{bpp-ll} satisfy the coupling constraints. This is in line with
the interpretation that the follower, i.e., the optimizer
of~\eqref{bpp-ll}, is an adversary to the leader, i.e., the optimizer
of~\eqref{bpp}, in the pessimistic case.

Problem~\eqref{bpp} is known to be strongly
NP-hard~\cite{hansen92}. Essentially, one can model implicit binarity
constraints for the upper level variables~$x$ by an appropriate lower
level construction~\cite{audet97}, so that every single-level binary
linear optimization problem can be represented in the
form~\eqref{bpp}: for each binary variable~$x_i$, a new variable~$y_i$
together with constraints $y_i\le x_i$ and~$y_i\le 1-x_i$ and the
objective $\max y_i$ are introduced on the lower level. This
implies~$y_i=\min\{x_i,1-x_i\}$ for every bilevel-feasible
solution. On the upper level, a constraint~$y_i=0$ is added,
then~$\min\{x_i,1-x_i\}=0$ and hence~$x_i\in\{0,1\}$. To avoid the
coupling constraint, one may also minimize~$y_i$ in the upper level
and add linear constraints~$x_i\in[0,1]$. Vicente et
al.~\cite{vicente94} show that even checking local optimality of a
given feasible solution to~\eqref{bpp} is NP-hard.

The complexity of~\eqref{bpp} is further investigated by
Deng~\cite{deng98}. Besides discussing hardness results and
efficiently solvable special cases, he claims that bilevel linear
optimization is NP-easy, i.e., can be Turing-reduced to some
NP-complete problem~\cite[Theorem~6.2]{deng98}. However, instead of a
proof, only a reference to an unpublished (and untitled) manuscript is
given. This result is related to, but weaker than the assertion that
the decision version of~\eqref{bpp} belongs to~NP.  To the best of our
knowledge, no rigorous proof of NP-membership has appeared in the
literature yet, a gap that we will fill in Section~\ref{sec_np} below.

One of the classical approaches for solving Problem~\eqref{bpp} in
practice is to reformulate it as a mixed integer linear optimization
problem, using the Karush-Kuhn-Tucker (KKT) conditions to model
optimality for the lower level problem~\cite{fortuny-amat81}. More
precisely, in a first step, one can rewrite~\eqref{bpp} equivalently
as
\begin{equation}\label{dual}\tag{BLP-KKT}
\begin{array}{ll}
  \min & \, c^\top x+d^\top y\\
  \textnormal{~s.t.} & Ax+By = a,~x\ge 0\\
  & Tx+Wy=h,~y\ge 0\\
  & W^\top\lambda \le q\\
  & y^\top (q-W^\top\lambda)=0\;.
\end{array}
\end{equation}
The main difficulty here is the non-linear complementarity constraint
$y^\top (q-W^\top\lambda)=0$. It is typically linearized by
introducing binary variables~$z_i\in\{0,1\}$ and constraints
\begin{equation}\label{bounds}
  y_i\le M_p(1-z_i),~(q-W^\top\lambda)_i\le M_dz_i
\end{equation}
for sufficiently large~$M_p,M_d\in\R$. This ensures that
either~$y_i=0$ or~$(q-W^\top\lambda)_i=0$ for all~$i$.

In general, choosing~$M_p$ and~$M_d$ correctly is not a trivial task:
too small values of~$M_p$ and~$M_d$ may cut off lower level optimal
solutions of~\eqref{bpp}, as~\eqref{bounds} implies that~$y_i\le M_p$
and~$(q-W^\top\lambda)_i\le M_d$ for all~$i$. This problem is
highlighted in~\cite{pineda19}. In fact, Kleinert et
al.~\cite{kleinert20} recently showed that it is a co-NP-complete
problem to decide whether given values for~$M_p$ and~$M_d$ are
\emph{bilevel-correct}, i.e., yield an equivalent reformulation
of~\eqref{bpp}. They claim that their results \emph{``imply that there
is no hope for an efficient, i.e., polynomial-time, general-purpose
method for verifying or computing a correct big-M in bilevel
optimization unless P\,$=$\,NP''}~\cite{kleinert20}. However, it is
not true that the second part of this claim follows from the first
one: it may still be possible to compute \emph{some}
bilevel-correct~$M_p$ and~$M_d$ in polynomial time, even if verifying
bilevel-correctness for \emph{given}~$M_p$ and~$M_d$ is not
possible. Stated differently, the result of Kleinert et al.\ rules out
that the \emph{smallest possible} bilevel-correct~$M_p$ and~$M_d$ can
be found efficiently, if P\;$\neq$\;NP, but it does \emph{not} rule
out that there exist bilevel-correct~$M_p$ and~$M_d$ that can be
computed efficiently, and thus in particular have polynomial encoding
length -- even though they may be much larger than necessary.

In fact, devising an approach for constructing bilevel-correct ~$M_p$
and~$M_d$ in polynomial time is closely related to showing that (the
decision version of) Problem~\eqref{bpp} is contained in NP, i.e.,
that there exist certificates of polynomial size. It is quite natural
to define the certificate as a feasible solution to either \eqref{bpp}
or the KKT reformulation~\eqref{dual}. However, it is not obvious that
there exists such a solution of polynomial encoding length. In the
following, we first give a short proof that the decision version of
Problem~\eqref{bpp} indeed belongs to NP (Section~\ref{sec_np}) and
then show how to compute bilevel-correct~$M_p$ and~$M_d$ in polynomial
time (Section~\ref{sec_bigm}).

\section{Membership in NP}\label{sec_np}

We first show that the decision version of Problem~\eqref{bpp} belongs
to NP. For a precise definition of the class NP and the corresponding
certificates, as well as encoding lengths and other
complexity-theoretic concepts, we refer to~\cite{groetschel93}.  As is
common for optimization problems, we formally define the decision
version of Problem~\eqref{bpp} as follows:

\medskip

\underbar{Decision version of~\eqref{bpp}, optimistic scenario}\\[-4ex]
\begin{tabbing}
  Given: \, \= Problem data $c\in\Q^n$, $d,q\in\Q^m$, $A\in\Q^{k\times n}$, $B\in\Q^{k\times m}$, $a\in\Q^k$,\\
  \> $T\in\Q^{r\times n}$, $W\in\Q^{r\times m}$, $h\in\Q^r$, and a number $\alpha\in\Q$.\\[1.5ex]
  Task: \> Decide whether there exists~$\bar x\ge 0$ such that~\eqref{bpp-ll} has an optimal solution and such that\\
    \>at least one optimal solution~$\bar y$ of~\eqref{bpp-ll} satisfies~$A\bar x+B\bar y=a$ and~$c^\top \bar x+d^\top \bar y \le \alpha$.
\end{tabbing}
Note that this decision version of~\eqref{bpp} is equivalent to the
problem of determining whether~\eqref{bpp} admits any feasible
solution, since the constraint~$c^\top x+d^\top y \le \alpha$ can be
seen as an additional coupling constraint.

The idea of the following proof is to reformulate the feasibility of a
given solution to Problem~\eqref{bpp} that also satisfies~$c^\top
x+d^\top y \le \alpha$ by a system of linear constraints. However, in
order to enforce the optimality of the lower level solution, we need
to choose an optimal basis for~\eqref{bpp-ll}, which will then form
the desired certificate. In other words, the NP-hardness
of~\eqref{bpp} is only due to the required choice of the optimal lower
level basis.  An important technical ingredient in the following proof
is that the reduced costs of a given basis for~\eqref{bpp-ll} do not
depend on~$\bar x$.
\begin{theorem}\label{theo_np}
  The decision version of Problem~\eqref{bpp} in the optimistic scenario belongs to NP.
\end{theorem}
\begin{proof}
  Assume that a yes-instance is given, i.e., that there exist feasible
  $x'\in\R^n$ and $y'\in\R^m$ for~\eqref{bpp} with~$c^\top x'+d^\top
  y' \le \alpha$. Let~$\cal B$ be any optimal basis for the lower
  level problem for~$x'$. Then~$\cal B$ has non-negative reduced costs
  w.r.t.~$q$ and~$(x',y')$ together satisfy the linear constraints
  \begin{subequations}\label{eq_cond}
    \begin{align}
      & c^\top x+d^\top y\le \alpha\\
      & Ax+By=a,~x\ge 0\label{eq_cond:ul}\\
      & Tx+Wy=h,~y\ge 0\label{eq_cond:ll}\\
      & q^\top y=q_{\cal B}^\top W_{\cal B}^{-1}(h-Tx)\label{eq_cond:opt}\\
      & W_{\cal B}^{-1}(h-Tx)\ge 0\;.\label{eq_cond:feas}
    \end{align}
  \end{subequations}
  The constraint \eqref{eq_cond:opt} follows from the optimality
  of~$y'$, since~$W_{\cal B}^{-1}(h-Tx')$ is the basic part of the
  optimal solution corresponding to~$\cal B$. The constraint
  \eqref{eq_cond:feas} expresses that~$\cal B$ is a feasible basis for
  the lower level problem for~$x'$.

  Conversely, assume that there exists a basis~$\cal B$ of~$W$ with
  non-negative reduced costs w.r.t.~$q$ such that~\eqref{eq_cond} is
  satisfied by some~$\bar x$ and~$\bar y$.  By~\eqref{eq_cond:feas},
  the basis~$\cal B$ is feasible for the lower level problem for~$\bar
  x$, hence also optimal because of the non-negative reduced costs
  w.r.t.~$q$. It thus follows from~\eqref{eq_cond:opt} that~$\bar y$
  is an optimal lower level solution for~$\bar x$. Together with the
  remaining constraints in~\eqref{eq_cond}, we derive that~$\bar x$
  and~$\bar y$ form a feasible solution to~\eqref{bpp} with~$c^\top
  \bar x+d^\top \bar y\le \alpha$.

  A polynomial-size certificate for a yes-instances thus consists of a
  basis~${\cal B}\subseteq\{1,\dots,r\}$ with non-negative reduced
  costs w.r.t.~$q$ such that the linear system \eqref{eq_cond} is
  solvable. The latter can be tested in time polynomial in the
  coefficients of~\eqref{eq_cond}, which are in turn polynomial in the
  coefficients of the original problem~\eqref{bpp}.
\end{proof}
Note that the basis~$\cal B$ in the above proof does not necessarily
correspond to the solution~$y'$ or~$\bar y$, if the lower level
solution is not unique. In fact, in the optimistic setting considered
here, it may happen that no basic solution of the follower's problem
is feasible. The constraints~\eqref{eq_cond:ll}
and~\eqref{eq_cond:opt} in the proof are thus needed to model the
optimal face of the lower level problem.  As a (somewhat pathological)
example, which does not contain any leader variables, consider
\begin{equation*}
\begin{array}{ll}
  \min &  y\\
  \textnormal{~s.t.} & y = 1\\
  & y \in \text{argmin}\;\{ 0 \mid \bar y\le 2,~\bar y\ge 0\}=[0,2]\;.
\end{array}
\end{equation*}
The optimal value of this problem is~$1$, since~$y=1$ is the unique
bilevel-feasible solution. However, the only two basic solutions of
the follower's subproblem are~$y=0$ and~$y=2$, which do not satisfy
the upper level constraint $y=1$. On the other hand, if there are no
upper level coupling constraints, i.e., if~$B=0$, it is easy to see
that there always exists an optimal lower level \emph{basic} solution;
see also~\cite[Theorem~2.1]{Dempe}.

We emphasize that the proof of Theorem~\ref{theo_np} does not make any
assumptions on the feasibility or boundedness of the lower level
problem~\eqref{bpp-ll}: since we start with feasible~$x'$ and~$y'$, we
know by definition that~$y'$ is an optimal lower level response to the
upper level decision~$x'$. Hence the lower level problem cannot be
unbounded or infeasible for the given~$x'$. Moreover, the second part
of the proof shows that a certificate cannot exist if there is no
upper level decision~$\bar x$ such that the lower level
problem~\eqref{bpp-ll} has an optimal solution.
    
We now turn our attention to the pessimistic case, where the follower
tries to violate the coupling constraints and, if not possible,
returns a worst-possible solution for the upper level problem. We
first define the decision version again:

\medskip

\underbar{Decision version of~\eqref{bpp}, pessimistic scenario}\\[-4ex]
\begin{tabbing}
  Given: \, \= Problem data $c\in\Q^n$, $d,q\in\Q^m$, $A\in\Q^{k\times n}$, $B\in\Q^{k\times m}$, $a\in\Q^k$,\\
  \> $T\in\Q^{r\times n}$, $W\in\Q^{r\times m}$, $h\in\Q^r$, and a number $\alpha\in\Q$.\\[1.5ex]
  Task: \> Decide whether there exists~$\bar x\ge 0$ such that~\eqref{bpp-ll} has an optimal solution and such that\\
    \>\emph{all} optimal solutions~$\bar y$ of~\eqref{bpp-ll} satisfy~$A\bar x+B\bar y=a$ and~$c^\top \bar x+d^\top \bar y \le \alpha$.
\end{tabbing}
The universal quantifier in this formulation creates some additional
difficulty: for a yes-instance, we have to make sure that~$\bar x$ is
chosen such that all optimal solutions of the lower level problem
satisfy the coupling constraints and $c^\top \bar x+d^\top \bar y \le
\alpha$. More abstractly, we thus have to enforce an inclusion
relation between two polyhedra that are both parametrized by~$\bar
x$. In the following proof, we need a second basis in our certificate
to achieve this.
\begin{theorem}\label{theo_np_pess}
  The decision version of Problem~\eqref{bpp} in the pessimistic
  scenario belongs to NP.
\end{theorem}
\begin{proof}
  Assume that a yes-instance is given. Then, in particular, some~$\bar
  x\ge 0$ exists such that~\eqref{bpp-ll} has an optimal
  solution. This means that there exists a basis~$\cal B$ of~$W$ with
  non-negative reduced costs w.r.t.~$q$ such that the corresponding
  basic solution is feasible, i.e., such that~$\bar x$ satisfies
  \begin{equation}\label{eq_cond_pess}
    W_{\cal B}^{-1}(h-Tx)\ge 0,~x\ge 0\;.
  \end{equation}
  Moreover, every optimal solution~$y$ of~\eqref{bpp-ll} satisfies
  both~$A\bar x+By=a$ and~$c^\top \bar x+d^\top y \le \alpha$.  In
  other words, we have that all solutions~$y$ of
  \begin{equation}\label{eq:opt}
    T\bar x+Wy=h,~q^\top y=q_{\cal B}^\top W_{\cal B}^{-1}(h-T\bar x),~y\ge 0
  \end{equation}
  also satisfy
  $$A\bar x+By=a,~c^\top\bar x+d^\top y \le \alpha\;.$$
  In particular, we deduce
  \begin{equation}\label{eq:inctest}
    \max_{\eqref{eq:opt}} \; d^\top y \le -c^\top \bar x+\alpha\;,
  \end{equation}
  and this maximization problem has an optimal solution since it is
  bounded and~\eqref{eq:opt} is feasible. Since~\eqref{eq:opt} is in
  standard form, there exists a
  basis~$\hat {\cal B}$ for~$\hat W:=\binom{W}{q^\top}$
  with non-positive reduced costs w.r.t.~$d$ such
  that~$\bar x$ satisfies the linear constraints
  \begin{subequations}\label{eq_cond_inc0}
    \begin{align}
      & \textstyle d_{\hat{\cal B}}^\top \hat W_{\hat{\cal B}}^{-1}\binom{h-Tx}{q_{\cal B}^\top W_{\cal B}^{-1}(h-Tx)} \le -c^\top x+\alpha \label{eq_cond_inc0:f1}\\[1ex]
      & \textstyle \hat W_{\hat{\cal B}}^{-1}\binom{h-Tx}{q_{\cal B}^\top W_{\cal B}^{-1}(h-Tx)}\ge 0\;.\label{eq_cond_inc0:f2}
    \end{align}
  \end{subequations}
  Indeed, constraint~\eqref{eq_cond_inc0:f2} states that the
  basis~$\hat{\cal B}$ is feasible for~\eqref{eq:opt}, while
  constraint~\eqref{eq_cond_inc0:f1} states that the corresponding
  basic solution~$y$ satisfies the constraint~$c^\top \bar x+d^\top
  y\le\alpha$.

  Proceeding to the coupling constraints~$A\bar x+By=a$, let~$b^i$
  and~$a^i$ denote the $i$-th row of~$A$ and~$B$, respectively, and
  let~$a_i$ denote the $i$-th entry of~$a$. Then, for
  all~$i=1,\dots,k$, since~$b^iy=-a^i\bar x+a_i$ for all~$y$
  satisfying~\eqref{eq:opt}, there exist bases~$\hat{\cal B}_i^+$
  and~$\hat{\cal B}_i^-$ for~$\hat W$ that have
  non-negative and non-positive reduced costs, respectively, w.r.t.~$(b^i)^\top$ such that~$\bar x$ satisfies the linear constraints
  \begin{subequations}\label{eq_cond_inci}
    \begin{align}
      & \textstyle b^{i}_{\hat{\cal B}_i^+} \hat W_{\hat{\cal
        B}_i^+}^{-1}\binom{h-Tx}{q_{\cal B}^\top W_{\cal
        B}^{-1}(h-Tx)} \ge -a^i x+a_i \ge
       \textstyle b^{i}_{\hat{\cal B}_i^-} \hat W_{\hat{\cal
        B}_i^-}^{-1}\binom{h-Tx}{q_{\cal B}^\top W_{\cal
        B}^{-1}(h-Tx)}   \\[1ex]
      & \textstyle \hat W_{\hat{\cal
        B}_i^+}^{-1}\binom{h-Tx}{q_{\cal B}^\top W_{\cal
        B}^{-1}(h-Tx)}, \hat W_{\hat{\cal
        B}_i^-}^{-1}\binom{h-Tx}{q_{\cal B}^\top W_{\cal
        B}^{-1}(h-Tx)} \ge 0\;.
    \end{align}
  \end{subequations}
  To sum up, for any yes-instance, there exist a basis~$\cal B$ of~$W$
  with non-negative reduced costs w.r.t.~$q$, a basis~$\hat{\cal B}$
  of~$\hat W$ with non-positive reduced costs w.r.t~$d$, and
  bases~$\hat{\cal B}_i^+$ and~$\hat{\cal B}_i^-$ of~$\hat W$ with
  non-negative and non-positive reduced costs, respectively,
  w.r.t.~$(b^i)^\top$, for~$i=1,\dots,k$, such that all linear
  constraints on~$x$ in~\eqref{eq_cond_pess}, \eqref{eq_cond_inc0},
  and~\eqref{eq_cond_inci} can be satisfied simultaneously.

  Conversely, if such bases exist, the instance is a
  yes-instance. Indeed, by~\eqref{eq_cond_pess} the corresponding
  feasible solution~$\bar x$ admits an optimal response.
  From~\eqref{eq_cond_inc0} it follows that
  $$\max_{\eqref{eq:opt}} \; d^\top y=\textstyle d_{\hat{\cal B}}^\top
  \hat W_{\hat{\cal B}}^{-1}\binom{h-T\bar x}{q_{\cal B}^\top W_{\cal
      B}^{-1}(h-T\bar x)} \le -c^\top \bar x+\alpha\;,$$ as~$\hat{\cal
    B}$ is a feasible and hence optimal basis for the given
  maximization problem. Thus~$c^\top\bar x+d^\top y\le\alpha$ for
  all~$y$ feasible for~\eqref{eq:opt}.  From~\eqref{eq_cond_inci} we
  similarly derive~$A\bar x+By=a$ for all such~$y$. The result thus
  follows with the certificate consisting of the bases~${\cal B}$,
  $\hat{\cal B}$, and~$\hat{\cal B}_i^+,\hat{\cal B}_i^-$
  for~$i=1,\dots,k$.
\end{proof}
We emphasize that, different from the proof of Theorem~\ref{theo_np},
the linear constraints constructed in the proof of
Theorem~\ref{theo_np_pess} do not contain the lower level variables,
they only restrict the upper level decision. This is in line with the
fact that, in the pessimistic case, the lower level variables appear
in a universal instead of an existential quantifier. In order to deal
with this ``adversarial'' quantifier, the additional~$1+2k$ bases are
needed in the certificate.

One may be tempted to simplify the above proof by replacing the
maximization problem in~\eqref{eq:inctest} by its dual and then
resolving the resulting minimization problem by introducing additional
variables. However, this approach would lead to a non-linear model,
because the new variables would be multiplied by the variables~$x$
appearing on the right hand side of~\eqref{eq:opt}. For the resulting
quadratic system, testing feasibility would not be trivially
polynomial any more.

\section{Computation of bilevel-correct bounds}\label{sec_bigm}

For sake of simplicity, we focus on the optimistic scenario in the
remainder of this paper. From an abstract point of view, it already
follows from Theorem~\ref{theo_np} that there exists a polynomial-time
algorithm that reformulates any bilevel linear optimization problem
into an equivalent mixed-integer linear optimization problem. Indeed,
as mixed-integer linear optimization is NP-hard, any decision problem
in NP can be polynomially reduced to it. We now show that this task
can actually be achieved by the KKT-reformulation~\eqref{dual}, after
deriving bilevel-correct bounds.  In fact, the proof of
Theorem~\ref{theo_np} already implies that we can compute a bound on
the entries of~$x$ and~$y$ in polynomial time such that this bound
does not cut off all optimal solutions. We now make this bound more
explicit and extend this statement to the dual variables.  For the
convenience of the reader, we first report the following elementary
result. We will denote by~$\mx X$ the maximal absolute value of any
entry in an integer matrix or vector~$X$.
\begin{lemma}\label{ppoly2}
  Let~$P=\{x\in\R^n\mid Ax=b,\,x\ge 0\}$ be a polyhedron
  with~$A\in\Z^{m\times n}$ and $b\in\Z^m$. Let~$\bar x$ be a vertex
  of~$P$. Then $\bar x_i\le m!\, \mx b \, \mx A^{m-1}$ for
  all~$i=1,\dots,n$.
\end{lemma}
\begin{proof}
  By removing redundant constraints, we may assume that~$A$ has full
  row rank.  Let~$\cal B$ be a basis yielding~$\bar x$. Then~$\bar
  x_i=0$ for~$i\not\in{\cal B}$ and $x_{\cal B}=A_{\cal
    B}^{-1}b$. Hence, for~$i\in{\cal B}$, we have $|\bar x_i|\le
  |\det(A^{i}_{\cal B})|$ by Cramer's rule, where~$A^{i}_{\cal B}$
  arises from~$A_{\cal B}$ by replacing column~$i$ by~$b$; note
  that~$|\det(A_{\cal B})|\ge 1$ by integrality and regularity
  of~$A_{\cal B}$. Hence
  $$\textstyle |\bar x_i|\le |\sum_{\sigma\in
    S_m}\text{sgn}(\sigma)\prod_{j=1}^m (A^{i}_{\cal
    B})_{j,\sigma(j)}|\le \sum_{\sigma\in S_m}\prod_{j=1}^m
  |(A^{i}_{\cal B})_{j,\sigma(j)}|\;,$$
  which directly implies the result.
\end{proof}
Note that the number~$m!$ appearing in the bound above has polynomial
encoding length in~$m\in\N$, since~$m!\le m^m$ and the latter has
encoding length at most~$m(\lceil\log (m+1)\rceil+1)$. In particular,
the naive algorithm to compute~$m!$ runs in polynomial time, but there
exist faster algorithms performing this task in~$O(m\log^2 m)$
time~\cite{harvey21,schoenhage94}.
\begin{theorem}\label{theo_md}
  Bilevel-correct values for~$M_p$ and~$M_d$ can be computed in
  polynomial time from the data of Problem~\eqref{bpp}.
\end{theorem}
\begin{proof}
  We may assume that~\eqref{bpp} admits an optimal
  solution~$x',y'$. As in the proof of Theorem~\ref{theo_np}, we
  choose~$\cal B$ as an optimal basis of the lower level problem
  for~$x'$. Let~$(\bar x,\bar y,\bar z)$ be an optimal vertex of the
  polyhedron defined by
  \begin{equation}\label{eq_cond_2}
    \begin{array}{ll}
      Ax+By=a\\
      Tx+Wy=h\\
      q_{\cal B}^\top W_{\cal B}^{-1}Tx+q^\top y=q_{\cal B}^\top W_{\cal B}^{-1}h\\
      W_{\cal B}^{-1}Tx+z = W_{\cal B}^{-1}h\\
      x,y,z\ge 0\;,
    \end{array}
  \end{equation}
  with respect to the objective function~$c^\top x+d^\top y$.
  Then~$\bar\lambda:=W_{\cal B}^{-\top}q_{\cal B}$ is a dual optimal
  solution for the lower level problem for~$\bar x$.  In particular,
  we have $(h-T\bar x)^\top\bar\lambda=q^\top\bar y$, which implies
  the complementarity constraint~$\bar y^\top (q-W^\top\bar
  \lambda)=0$. In summary, the constructed solution consisting
  of~$\bar x$, $\bar y$, and~$\bar \lambda$ is feasible and hence
  optimal for~\eqref{dual}. It thus suffices to bound the entries
  of~$\bar y$ and~$q-W^\top\bar\lambda$ in terms of the problem data
  of~\eqref{bpp}, independently of the choice of~$\cal B$.  By scaling
  all constraints and objective functions, we may assume that all
  coefficients in~\eqref{bpp} are integer. This can be done in
  polynomial time, as it suffices to scale by the product of all
  numerators of these coefficients. In particular, the encoding length
  of the scaled instance is polynomial in the original encoding
  length.
  
  To obtain a valid choice of~$M_d$, we need~$M_d\ge
  (q-W^\top\bar\lambda)_i=(q-W^\top W_{\cal B}^{-\top}q_{\cal B})_i$
  for all bases~$\cal B$ of~$W$. For $i\in\cal B$ we
  have~$(q-W^\top\bar\lambda)_i=0$ by construction, while the
  remaining entries contain the reduced costs of~$\cal B$. It thus
  suffices to compute a bound on the reduced costs of any basis of~$W$
  in terms of~$W$ and~$q$. For~$i\in{\cal N}$, we have
  $$|q_i-W_{\cdot i}^\top W_{\cal B}^{-\top}q_{\cal B}| \le
  |q_i|+|W_{\cdot i}^\top W_{\cal B}^{-\top}q_{\cal B}| \le
  |q_i|+\textstyle\sum_{j\in{\cal B}} |W_{ij}||(W_{\cal B}^{-\top}q_{\cal B})_{j}|\;.$$
  Now, since~$W$ is assumed to be integer, we have $|\det(W_{\cal
    B}^\top)|\ge 1$, so by Cramer's rule, each entry of~$W_{\cal
    B}^{-\top}q_{\cal B}$ is bounded by $\det(\bar W)$, where~$\bar W$
  arises from~$W_{\cal B}$ by replacing some column by~$q_{\cal
    B}$. Hence $|(W_{\cal B}^{-\top}q_{\cal B})_{j}|\le r!\, \mx q\mx
  W^{r-1}$ and we can choose~$M_d$ as
  $$\mx q+r \mx W r!\,\mx q \mx W^{r-1}=\mx q (1+r!\,r \mx W^r)\;.$$

  To determine a bilevel-correct value for~$M_p$, it suffices to bound
  the entries of~$y$ in all vertices of~\eqref{eq_cond_2}. Using
  Lemma~\ref{ppoly2}, we can choose~$M_p$ as $\ell!\,\mx f\mx
  L^{\ell-1}$ where~$\ell:=k+r+1+m$ is the number of equations
  in~\eqref{eq_cond_2} and
  \begin{eqnarray*}
    \mx L  & := & \max\{\mx A,\mx B,\mx T,\mx W,\mx{q_{\cal B}^\top W_{\cal B}^{-1}T},\mx q,\mx{W_{\cal B}^{-1}T},1\}\\
    \mx f & := & \max\{\mx a,\mx h,\mx{q_{\cal B}^\top W_{\cal B}^{-1}h},\mx{W_{\cal B}^{-1}h}\}\;.
  \end{eqnarray*}
  We further have
  \begin{eqnarray*}
    \mx{W_{\cal B}^{-1}h} & \le & r!\, \mx W^{r-1} \mx h\\
    \mx{q_{\cal B}^\top W_{\cal B}^{-1}h} & \le & r!\,r \mx q \mx W^{r-1} \mx h \\
    \mx{W_{\cal B}^{-1}T} & \le & r!\,\mx W^{r-1} \mx T\\
    \mx{q_{\cal B}^\top W_{\cal B}^{-1}T} & \le & r!\,r \mx q \mx W^{r-1} \mx T\;.
  \end{eqnarray*}
  Altogether, we can compute bilevel-correct bounds~$M_d$ and~$M_p$ in
  polynomial time from the coefficients of~\eqref{bpp}.
\end{proof}
Note that the bound~$M_d$ derived in the proof of
Theorem~\ref{theo_md} only depends on~$W$ and~$q$, while~$M_p$ depends
on all problem data, including the upper level coefficients.  If~$r\ge
1$ and none of~$q,W,T,h$ is all zero, the bounds can be simplified to
\begin{eqnarray*}
  \mx L  & \le & \max\{\mx A,\mx B,r!\,r \mx q \mx W^{r-1} \mx T\}\\
  \mx f & \le & \max\{\mx a,r!\,r \mx q \mx W^{r-1} \mx h\}\;.
\end{eqnarray*}

Both bounds constructed in Theorem~\ref{theo_md} have polynomial
encoding length. However, their values are exponential in general. The
latter cannot be avoided, even without the bilevel structure. As a
simple example, consider constraints~$y_1=1$ and~$2y_i-y_{i+1}=0$
for~$i=1,\dots,m-1$. Then all coefficients have size at most~2, but in
the unique feasible solution, the value of~$y_n$ is~$2^{n-1}$.

\begin{remark}
  From the proof of Theorem~\ref{theo_np} it can be argued that the
  NP-hardness of~\eqref{bpp} is due to the exponential number of bases
  of the lower level problem. More precisely, even though it is
  possible to efficiently find an optimal lower level basis for any
  fixed upper level decision, since this reduces to solving a linear
  optimization problem, it is not possible to efficiently decide
  whether there exists a basis~$\cal B$ that yields a feasible
  system~\eqref{eq_cond}, unless~P\,$=$\,NP. Otherwise, it would
  follow from the proof of Theorem~\ref{theo_np} that~\eqref{bpp} is
  tractable. In terms of the bounds~$M_d$ and~$M_p$,
  Theorem~\ref{theo_np} implies that we can determine bilevel-correct
  values by considering all possible lower level bases, which is the
  core idea of the proof of Theorem~\ref{theo_md}. If, however, the
  task is to test bilevel-correctness of given values of~$M_d$
  and~$M_p$ as in~\cite{kleinert20}, we need to decide whether there
  exists a basis~$\cal B$ such that some feasible solution
  of~\eqref{eq_cond}, with~$\alpha$ being the optimal value
  of~\eqref{bpp}, has a small enough encoding length. As shown
  in~\cite{kleinert20}, the latter problem is NP-complete.  Both
  observations are related to the hardness of the problem~OVP, which
  was shown in~\cite{fukuda97} and used in the main proof
  of~\cite{kleinert20}: while it is easy to find an optimal basis for
  a given linear optimization problem, if one exists, it is hard to
  test whether an (unbounded) linear optimization problem has a basic
  solution exceeding a certain threshold. From an abstract point of
  view, the common difficulty of~\eqref{bpp} and OVP is thus to
  determine a suitable basis for the respective task, out of
  exponentially many candidates.
\end{remark}

\paragraph*{Acknowledgments}

This work was inspired by discussions held at Dagstuhl Seminar 22441
``Optimization at the Second Level''. We thank the anonymous reviewer
for many valuable comments that helped to significantly improve the
paper.

\newpage

\bibliographystyle{plain}

\end{document}